\documentclass{article}
\usepackage[letterpaper, margin=1in]{geometry}
\usepackage{amsmath}
\usepackage{amssymb}
\usepackage{listings}
\usepackage{courier}
\usepackage{mdframed}
\usepackage{fancyhdr}
\usepackage{graphicx}
\graphicspath{{images/}}
\usepackage[utf8]{inputenc}
\usepackage[english]{babel}
\usepackage{amsthm}
\usepackage[T1]{fontenc}
\usepackage{tikz-cd} 
\usepackage{enumitem}
\usepackage{tikz}
\usepackage{etoolbox}
\newcommand{\circled}[2][]{%
  \tikz[baseline=(char.base)]{%
    \node[shape = circle, draw, inner sep = 0.5pt]
    (char) {\phantom{\ifblank{#1}{#2}{#1}}};%
    \node at (char.center) {\makebox[0pt][c]{#2}};}}
\robustify{\circled}

\usepackage{listings}
\usepackage{color}

\definecolor{dkgreen}{rgb}{0,0.6,0}
\definecolor{gray}{rgb}{0.5,0.5,0.5}
\definecolor{mauve}{rgb}{0.58,0,0.82}

\theoremstyle{plain}
\newtheorem*{thm}{Theorem}
\newtheorem*{lem}{Lemma}
\newtheorem*{prop}{Proposition}

\newtheorem*{eg}{Example}

\theoremstyle{remark}

\newtheorem*{essay}{Essay}

\usepackage{color}

\usepackage{color,soul}
\usepackage{enumitem}
 
\begin{document}

\begin{titlepage}
    \begin{center}
        \vspace*{1cm}
        \Huge
        \textbf{A Report on Subobject Classifiers and Monads}
        \vspace{3.0cm}
        \Large
        
        \textbf{Arnold Tan Junhan} 
        \vspace{2.0cm}
        \Large
        
        \textbf{Michaelmas 2018 Mini Projects: Categories Proofs and Processes}
        \vfill
        \vspace{0.8cm}
        \Large
        University of Oxford   
    \end{center}
\end{titlepage}
\tableofcontents
%we should have consistent capitalisation of section names
\section{The subobject classifier and its logical applications}
\begin{essay}[1(a)]

\noindent In a given category $\mathcal{C}$, let us declare that a monic $f: a \rightarrowtail d$ is \textbf{contained} in a monic $g: b \rightarrowtail d$, if there is a map $h: a \rightarrowtail b$ such that $gh=f$.
\noindent We declare two monics $f,g$ into $d$ two be \textbf{equivalent}, written $f \simeq g$, if each is contained in the other. A \textbf{subobject} of $d$ is an equivalence class of monics into $d$, and the relation of containment gives us a poset $(Sub(d), \subseteq)$ on the subobjects of $d$.
\\

\noindent Let $\mathcal{C}$ be a category with a terminal object $1$. A \textbf{subobject classifier} of $\mathcal{C}$ is a $\mathcal{C}$-object $\Omega$ together with a $\mathcal{C}$-arrow $1 \xrightarrow{\textit{true}} \Omega$ that satisfies the following axiom:
\\

\textbf{($\Omega$-axiom):} For each monic $f: a \rightarrowtail d$ there is a unique $\mathcal{C}$-arrow $\chi_f: d \rightarrow \Omega$ making
\[ \begin{tikzcd}
a  \arrow{d}{!} \arrow[rightarrowtail]{r}{f}& d \arrow{d}{\chi_f} \\
1 \arrow{r}{true} & \Omega & \text{a pullback square.}
\end{tikzcd} 
\]
\noindent We say $\chi_f$ is the \textbf{character} of the subobject $f$, and often write $\top$ for $true$ and $true_a$ for $a \xrightarrow{!} 1 \xrightarrow{true} \Omega$.
\\

\noindent When they exist, subobject classifiers are unique up to (unique) isomorphism, and it is also easy to see that the assignment of $\chi_f$ to $f$ yields a 1-1 correspondence between subobjects of $d$ and arrow $d \rightarrow \Omega$ --- we have $Sub(d) \cong \mathcal{C}(d, \Omega)$ (at least, so long as we can take arbitrary pullbacks).
\\

\noindent An \textbf{(elementary) topos} $\mathcal{E}$ is a Cartesian closed category with a subobject classifier.
\\
\noindent For instance, $Set$ is a topos with subobject classifier $\top : 1 \xrightarrow{1} 2$. In Set, a monic is just a subset $a \subseteq d$, and we are familiar with the characteristic function $\chi_a : d \rightarrow \{ 0,1 \}$ that sends an element of $d$ to $1$ if it belongs to $a$, and $0$ otherwise. We have $\mathcal{P}(d) = Sub(d) \cong Set(d, 2)$, and furthermore there are operations of (set-theoretic) union $\cup$, intersection $\cap$ and complement $\lnot$ on $Sub(d)$.
\\ These operations turn $Sub(d)$ into a \textbf{Boolean algebra}: a complemented distributive lattice. In particular, we can use $Set(1,\Omega) = \{0,1\}$ and its operations of $\cup$, $\cap$ and $\lnot$ to model classical logic.
\\

\noindent That is the story in $Set$; subobject classifiers are important because they generalise this concept --- once we define on $\mathcal{E}(1, \Omega)$ categorical notions of union, intersection and negation that generalise their counterparts in the case of $Set$, we can model different flavours of logics on different topoi! Some of these logics will not even be classical --- that is, they will differ from the logic of $Set$, and $Sub(d)$ need not be a Boolean algebra. The algebra of subobjects in such a topos will only be a \textit{Heyting algebra}.
\\ Let us see how the whole idea of `modelling logic with topoi' works, starting with classical logic.
\\

\noindent Recall that in defining the formal language \textbf{PL} (propositional logic), we have the set $\Phi_0 := \{\pi_0, \pi_1, \pi_2, \ldots \}$ of variables, and the set $\Phi := \{ \alpha : \alpha \text{ is a \textbf{PL}-sentence} \}$ of sentences. (See the Appendix for a review of the formal language \textbf{PL}, along with the axiom systems \textbf{CL} (classical logic) and \textbf{IL} (intuitionistic logic).)
\\

\noindent Recall also that every Boolean algebra $\mathbb{B}=(B, \sqsubseteq, \sqcap, \sqcup, ', 0, 1)$ has, by definition, operations of meet ($\sqcap$), join ($\sqcup$), and complement ($'$). We may additionally define an implication operation as $x \Rightarrow y := x' \sqcup y$.
\\ We now describe the semantics of \textbf{PL} in $\mathbb{B}$:
\\ A \textbf{$\mathbb{B}$-valuation} is a function $V: \Phi_0 \rightarrow B$. By the following rules, any valuation lifts uniquely to a function $V: \Phi \rightarrow B$:
\begin{itemize}
    \item[(a)] $V(\sim \alpha) = V(\alpha)'$; 
    \item[(b)] $V(\alpha \land \beta) = V(\alpha) \sqcap V(\beta)$; 
    \item[(c)] $V(\alpha \lor \beta) = V(\alpha) \sqcup V(\beta)$; 
    \item[(d)] $V(\alpha \supset \beta) = V(\alpha) \Rightarrow V(\beta)$.
\end{itemize}
We say a sentence $\alpha$ is \textbf{$\mathbb{B}$-valid}, and write $\mathbb{B} \models \alpha$, if for every $\mathcal{B}$-valuation $V$ we have $V(\alpha)=1$.
\\ Note that in any $\mathbb{B}$, $0$ and $1$ provide a copy of the Boolean algebra \textbf{2}, so already we see that $\mathbb{B} \models \alpha$ only if $\alpha$ is a \textbf{tautology}, i.e., \textbf{2} $\models \alpha$.
\\

\noindent We have the following \textbf{Soundness Theorem} for $\mathbb{B}$-validity: \textit{If $\vdash_{CL} \alpha$ then $\mathbb{B} \models \alpha$.}
\\ The converse also holds! That is, we have the following \textbf{Completeness Theorem}: \textit{If $\mathbb{B} \models \alpha$ then $\vdash_{CL} \alpha$.}
\\

\noindent In particular, if a sentence is valid in \textit{some} Boolean algebra, then it is a \textbf{CL}-theorem, so it is valid in \textit{every} Boolean algebra!
\\ This is a powerful result, and indeed, completeness is the harder direction of the two to prove.
\\

\noindent Next, let us fulfill the promise of defining truth-arrows in a topos --- categorical notions of $\cap$, $\cup$ etc. on the subobject classifier. Let $\mathcal{E}$ be a topos with classifier $\top : 1 \rightarrow \Omega$. Make the following definitions:
\begin{enumerate}
   \iffalse
    \item $\lnot: \Omega \rightarrow \Omega$ is the unique $\mathcal{E}$-arrow such that
\[ \begin{tikzcd}
1  \arrow{d} \arrow{r}{\bot}& \Omega \arrow{d}{\lnot} \\
1 \arrow{r}{\top} & \Omega & \text{is a pullback square, where $\bot$ is the character of $!: 0 \rightarrow 1$.}
\end{tikzcd} 
\]
In other words, $\lnot = \chi_\bot$.
\fi
    \item $\lnot: \Omega \rightarrow \Omega$ is the character of $\bot : 1 \rightarrow \Omega$, where $\bot$ is the character of $!: 0 \rightarrow 1$.
    \item $\cap: \Omega \times \Omega \rightarrow \Omega$ is the character of the product arrow $\langle \top,\top \rangle : 1 \rightarrow \Omega \times \Omega$.
    \item $\cup: \Omega \times \Omega \rightarrow \Omega$ is the character of the image of the arrow $[\langle \top_\Omega, 1_\Omega \rangle, \langle 1_\Omega, \top_\Omega \rangle ]: \Omega + \Omega \rightarrow \Omega \times \Omega$.
    \\ (The \textbf{image} of a map $f:a \rightarrow b$ in a topos is the smallest subobject of $b$ through which $f$ factors.)
    
    \item $\Rightarrow: \Omega \times \Omega \rightarrow \Omega$ is the character of $e: \circled{$\leq$} \rightarrowtail \Omega \times \Omega$, where $e$ is the equaliser of $\Omega \times \Omega \overset{\cap \ }{\underset{\pi_1  }\rightrightarrows} \Omega $.
\end{enumerate}

\noindent In $Set$, unpacking the definitions gives us the classical truth functions.
\\ For instance, $\Rightarrow: 2 \times 2 \rightarrow 2$ sends $(1,0)$ to $0$ and all other tuples to $1$.
\\

\noindent Now we can describe the semantics of interpreting propositional logic in any topos $\mathcal{E}$!
\\ A \textbf{truth value} in $\mathcal{E}$ is an arrow $1 \rightarrow \Omega$.
\\ An \textbf{$\mathcal{E}$-valuation} is a function $V: \Phi_0 \rightarrow \mathcal{E}(1, \Omega)$. Similarly to valuations on a Boolean algebra, any such function extends to all of $\Phi$ by the following rules:
\begin{itemize}
    \item[(a)] $V(\sim \alpha) = \lnot \circ V(\alpha)$.
    \item[(b)] $V(\alpha \land \beta) = \cap \circ \langle V(\alpha), V(\beta) \rangle$.
    \item[(c)] $V(\alpha \lor \beta) = \cup \circ \langle V(\alpha), V(\beta) \rangle$.
    \item[(d)] $V(\alpha \subset \beta) = \Rightarrow \circ \ \langle V(\alpha), V(\beta) \rangle$.
\end{itemize}
We say $\alpha$ is $\mathcal{E}$\textbf{-valid}, and write $\mathcal{E} \models \alpha$, if for every $\mathcal{E}$-valuation $V$ we have $V(\alpha) = \top : 1 \rightarrow \Omega$.
\\
\noindent Several questions immediately arise.
\\ For one, we might wonder if the notions of $\mathbb{B}$-valuations and $\mathcal{E}$-valuations are related. We shall give a better result at the end of this section, once we generalise the notion of a Boolean algebra $\mathbb{B}$ to that of \textit{Heyting algebra} $\mathbb{H}$, and define $\mathbb{H}$-valuations entirely analogously to $\mathbb{B}$-valuations. It will then be noted that $Sub(d)$ is always a Heyting algebra, and that an $\mathcal{E}$-valuation is precisely a $Sub(1)$-valuation, where $1 \in \mathcal{E}$ is the terminal object.
\\

\noindent Another question one might ask is how `compatible' our topos interpretation is with the system \textbf{CL}: namely, we ask if \textbf{CL} is sound and complete for $\mathcal{E}$-validity. (This was the case for $\mathbb{B}$-validity, i.e., when we interpreted propositional logic in a Boolean algebra.)
\\ It turns out that \textbf{CL} is complete but not sound for $\mathcal{E}$-validity: in any topos $\mathcal{E}$, every $\mathcal{E}$-valid sentence is derivable as a theorem in \textbf{CL}, but there exist topoi $\mathcal{E}$ in which some \textbf{CL}-theorems are not $\mathcal{E}$-valid.
\\ More precisely, the first eleven axioms of \textbf{CL} (see Appendix) are always $\mathcal{E}$-valid, so we are saying that in some topoi the twelth axiom $\alpha \lor \sim \alpha$ is not valid.
\\ The `correct' axiom system which captures $\mathcal{E}$-validity is the system \textbf{IL} (intuitionistic  logic), obtained simply by removing the twelfth axiom of \textbf{CL}, keeping all other axioms, and the single inference rule. In \textbf{IL}, tautologies such as $\alpha \lor \sim \alpha$ and $\sim \sim \alpha \supset \alpha$ are not derivable, so this is genuinely a different system than \textbf{CL}.
\\

\noindent A topos is \textbf{degenerate} if there is an arrow $1 \rightarrow 0$, or equivalently, if all its objects are isomorphic. A topos is \textbf{bivalent} if $\top$ and $\bot$ are its only truth values. A topos is \textbf{classical} if $[\top, \bot]: 1+1 \rightarrow \Omega$ is an isomorphism.
\\ These are just various measures of how much a topos `looks like' $Set$, which is a non-degenerate bivalent classical topos. We remark that for a bivalent topos $\mathcal{E}$, we do have \textbf{CL}-soundness for $\mathcal{E}$-validity.
\\

\noindent As examples, the category $Set^2$ of pairs of sets is a classical, non-bivalent topos. If \textbf{M} is a monoid, then the category \textbf{M-Set} of its actions is a topos, and this topos is classical iff \textbf{M} is a group.
\\ In particular \textbf{M$_2$-Set} is not classical, where \textbf{M$_2$} is the monoid $(\{0,1\}, \cdot)$ where $\cdot$ is usual integer multiplication. It is, however, bivalent, so by the above paragraph \textbf{M$_2$-Set} models all the \textbf{CL}-theorems.
\\

\noindent Let us next discuss how to turn $Sub(d)$ into a lattice, which we alluded to earlier.
\\ Let $\mathcal{E}$ be a topos, and $d \in \mathcal{E}$. Using the operations we have defined on we define some operations on $Sub(d)$:
\begin{enumerate}
    \item The complement of $f: a \rightarrowtail d$ (relative to $d$) is the subobject $-f : -a \rightarrowtail d$ whose character is $\lnot \circ \chi_f$.
    \item The intersection of $f: a \rightarrowtail d$ and $g: b \rightarrowtail d$ is the subobject $f \cap g : a \cap b \rightarrowtail d$ whose character is $\chi_f \cap \chi_g: = \cap \circ \langle \chi_f, \chi_g \rangle$.
    \item The union of $f: a \rightarrowtail d$ and $g: b \rightarrowtail d$ is the subobject $f \cup g : a \cup b \rightarrowtail d$ whose character is $\chi_f \cup \chi_g: = \cup \circ \langle \chi_f, \chi_g \rangle$.
    \item The subobject $f \Rrightarrow g : a \Rrightarrow b \rightarrowtail d$, for subobjects $f: a \rightarrowtail d$ and $g: b \rightarrowtail d$, is that whose character is $\chi_f \Rightarrow \chi_g: = \Rightarrow \circ \ \langle \chi_f, \chi_g \rangle$.
\end{enumerate}

\noindent Then $(Sub(d), \subseteq)$ is a bounded distributive lattice, with $\cap$ and $\cup$ above providing the meet and join operations. ($1_d$ and $0_d$ provide the unit and zero.) If this is complemented, then it is a Boolean algebra by definition. This is \textit{not} always the case: while $f \cap -f \simeq 0_d$ always holds, $f \cup -f \simeq 1_d$ need not. Let us say a topos is \textbf{Boolean} if for every $d \in \mathcal{E}$, $(Sub(d), \subseteq)$ is a Boolean algebra.
The following are equivalent:
\begin{enumerate}
    \item $\mathcal{E}$ is Boolean;
    \item $Sub(\Omega)$ is a Boolean algebra;
    \item $\mathcal{E}$ is classical;
    \item $\bot= - \top$ in $Sub(\Omega)$;
    \item $\lnot \circ \lnot = 1_\Omega$;
    \item in $Sub(\Omega)$, $f \Rrightarrow g \simeq -f \cup g$;
    \item in \textit{each} $Sub(d)$, $f \Rrightarrow g \simeq -f \cup g$.
\end{enumerate}
In particular, in a non-Boolean topos $\Rrightarrow$ behaves differently from a Boolean implication operator.
\\

\noindent The following equivalent conditions are weaker than the above:
\begin{enumerate}
    \item $\mathcal{E} \models \alpha$ iff $\vdash_{CL} \alpha$ for every $\alpha$;
    \item $\mathcal{E} \models \alpha \lor \sim \alpha$ for any $\alpha$;
    \item $Sub(1)$ is a Boolean algebra.
\end{enumerate}
These really are weaker conditions. For instance we have remarked that \textbf{M$_2$-Set} models every \textbf{CL}-theorem but is not classical.
\\

\noindent The slogan is \textit{topoi generalise sets}, so let us go further and define, where $f: a \rightarrowtail d$ is a subobject of $d$ in topos $\mathcal{E}$, $x: 1 \rightarrow d$ to be \textbf{an element of} $f$ if $x$ factors through $f$. Write this as $x \in f$. We always have, in $Sub(d)$, $x \in f \cap g$ iff $x \in f$ and $x \in g$. However, the property $$x \in -f \text{ iff } x \notin f$$ holds in every $Sub(d)$, iff $\mathcal{E}$ is bivalent. As for the property $$x \in f \cup g \text{ and } x \in f \text{ or } x \in g,$$
if this holds in every $Sub(d)$ we say $\mathcal{E}$ is \textbf{disjunctive}.
\\ We have the following characterisation:
$$\text{If $\mathcal{E}$ is Boolean and non-degenerate, then it is disjunctive iff it is bivalent.}$$

\noindent A topos is \textbf{extensional} if in $Sub(d)$ we always have
$$f \subseteq g \text{ iff whenever } x: 1\rightarrow d \text{ and } x \in f \text{, we have } x \in g.$$
That is, extensional topoi are those in which subobjects are determined by their elements. $Set$ is extensional.
\\

\noindent Let us say a bit more about non-Boolean topoi in general. A topos fails to be Boolean precisely when some $(Sub(d), \subseteq)$ fails to be Boolean.
\noindent In a lattice $\mathbb{L}=(L, \sqsubseteq)$, we say $c \in L$ is the \textbf{pseudo-complement} of $a \in L$ \textbf{relative to} $b \in L$, written $c= a \Rightarrow b$, if $c$ is the greatest element of $\{x \in L: a \sqcap x \sqsubset b \}$.
\\ If $a \Rightarrow b$ exists for all $a, b \in L$, we say $\mathbb{L}$ is a \textbf{relatively pseudo-complemented (r.p.c.) lattice}.
\\

\noindent Finally, a \textbf{Heyting algebra} is an r.p.c. lattice with zero.
\\ If $\mathbb{H} = (H, \sqsubseteq, \Rightarrow, 0)$ is a Heyting algebra, we may define the \textbf{pseudo-complement} $\lnot : H \rightarrow H$ as $\lnot a = a \Rightarrow 0$.
\\ We define an \textbf{$\mathbb{H}$-valuation} as a function $V: \Phi_0 \rightarrow H$. Once again, such a function extends to a function on sentences, using $\sqcap, \sqcup, \Rightarrow, \lnot$ to respectively interpret $\land, \lor, \supset, \sim$ in exactly the same way as with $\mathbb{B}$-valuations.
\\A sentence $\alpha$ is $\mathbb{H}$-valid if for all $\mathbb{H}$-valuations $V$, $V(\alpha) = 1$. $\alpha$ is \textbf{HA}-valid if it is valid in every Heyting algebra.
\\We have Soundness and Completeness!
$$\text{$\alpha$ is \textbf{HA}-valid iff $\vdash_{IL} \alpha$.}$$

\noindent The point is, although $(Sub(d), \subseteq)$ need not be a Boolean algebra, it is always a Heyting algebra. It can be verified that the r.p.c. is given by $\Rrightarrow$.
\\ Since the $\Omega$-axiom gave us $Sub(d) \cong \mathcal{E}(d, \Omega)$ (as sets), we may also consider the latter as a Heyting algebra.
\\ To our relief, the following holds:
$$\mathcal{E} \models \alpha \text{ iff } \mathcal{E} (1, \Omega) \models \alpha \text{ iff } Sub(1) \models \alpha .$$
This is because the unit of the Heyting algebra $\mathcal{E}(1, \Omega)$ is $\top : 1 \rightarrow \Omega$.
\\

\noindent Soundness of \textbf{IL} for $\mathcal{E}$-validity now follows immediately for its soundness for \textbf{HA}-validity:
\\ If $\vdash_{IL} \alpha$ then $\alpha$ is \textbf{HA}-valid, so $\mathcal{E}(1, \Omega) \models \alpha$, so $\mathcal{E} \models \alpha$.
\\ In fact \textbf{IL}-Completeness for $\mathcal{E}$-validity also holds: $\text{If $\alpha$ is valid on every topos, then } \vdash_{IL} \alpha.$
\\ The latter is proven using some additional theory on \textit{Kripke-style semantics}, in Goldblatt (2006).
\\

\noindent As a final remark, we note that higher-order logics can also be interpreted in topoi -- these are logics with quantifiers $\forall$, $\exists$. All this again exploits the Heyting algebra structure on $Sub(d)$, which we recall hinges on the existence of the wonderful subobject classifier.
\end{essay}

\section{The subobject classifier on a presheaf topos}

\begin{eg}[1(b)(i)]
The subobject classifier $\Omega$ of the presheaf topos given by $Set^{P^{op}}$, \\ where $P$ is the powerset of $\{1,2,3\}$ seen as a poset under inclusion, is described as follows.
\end{eg}

\noindent For a given object $a$ in $P$, let $S_a$ be the collection of all elements in $P$ contained in $a$,
$$S_a = \{b: b \subseteq a \}.$$
A $\textbf{crible}$ on $a$, or $a$-\textbf{crible}, is a downwards-closed subset $S$ of $S_a$, meaning whenever $b \in S$ and $c \subseteq b$, then $c \in S$.
\\ Then the subobject classifier is the functor $\Omega: P^{op} \rightarrow Set$ defined on objects by by
$$\Omega(a) = \{S: S \text{ is an } a \text{-crible}\}$$
and on maps by
$$\Omega(b \subseteq a) = (\Omega(a) \rightarrow \Omega(b), S \mapsto \{c \in S : c \subseteq b\}).$$
$$$$

\noindent (More generally, the subobject category on a functor category $[\mathfrak{C},Set]$ is described similarly using the dual notion of cribles, called $sieves$. This is explained in Goldblatt (2006).)
\\ In our example, there are twenty generalised truth values (arrows from the terminal object to $\Omega$).
\\

\noindent To see this, note that the terminal object in $Set^{P^{op}}$ is the functor that sends every element of $P^{op}$ to the terminal object of $Set$, which is just a singleton $1$. The (co)representable functor $H_{\{1,2,3\}} = P(-, \{1,2,3\})$ does this, since for any $a \subseteq \{1,2,3\}$ the set $P(a, \{1,2,3\})$ has precisely one element, given by $a \subseteq \{1,2,3\}$.
\\

\noindent Next, observe that
\begin{align*} 
Set^{P^{op}}( H_{\{1,2,3\}} , \ \Omega)
\cong  & \  \Omega(\{1,2,3\}) && \text{by the Yoneda Lemma.}\\
\end{align*}

\noindent This means that the generalized truth values are just elements of $\Omega(\{1,2,3\})$, i.e., $\{1,2,3\}$-cribles. There are twenty of these:

\begin{align*}
    & \ \emptyset , \\
    & \{\emptyset \}, \\
    & \{\emptyset, \{1\} \}, \qquad \{\emptyset, \{2\} \}, \qquad \{\emptyset, \{3\} \}, \\
    & \{\emptyset, \{1\}, \{2\} \}, \qquad \{\emptyset, \{2\}, \{3\} \}, \qquad \{\emptyset, \{1\}, \{3\} \}, \\
    & \{\emptyset, \{1\}, \{2\}, \{3\} \},\\
    & \{\emptyset, \{1\}, \{2\}, \{1,2\} \}, \qquad \{\emptyset, \{1\}, \{2\}, \{3\}, \{1,2\} \} \\
    & \{\emptyset, \{2\}, \{3\}, \{2,3\} \}, \qquad \{\emptyset, \{1\}, \{2\}, \{3\}, \{2,3\} \} \\
    & \{\emptyset, \{1\}, \{3\}, \{1,3\} \}, \qquad \{\emptyset, \{1\}, \{2\}, \{3\}, \{1,3\} \} \\
    & \{\emptyset, \{1\}, \{2\}, \{3\}, \{1,2\}, \{1,3\} \},\qquad \{\emptyset, \{1\}, \{2\}, \{3\}, \{1,2\}, \{2,3\} \},\qquad \{\emptyset, \{1\}, \{2\}, \{3\}, \{2,3\}, \{1,3\} \}, \\
    & \{\emptyset, \{1\}, \{2\}, \{3\}, \{1,2\} \{2,3\}, \{1,3\} \}, \\
    & \{\emptyset, \{1\}, \{2\}, \{3\}, \{1,2\} \{2,3\}, \{1,3\}, \{1,2,3\} \}. \\
\end{align*}

\noindent We have systemically listed these twenty $\{1,2,3\}$-cribles in ascending order of the size of a largest set contained in the crible: first we listed the empty crible, then we listed the singleton crible, then the cribles whose largest size of an element is one, then those whose largest size of an element is two, then the crible equal to the entire powerset $\mathbb{P}(\{1,2,3\})$ of $\{1,2,3\}$.
\\

\noindent Henceforth we will often write $\Omega$ to mean $\Omega(\{1,2,3\}) \cong \text{Sub}(H_{\{1,2,3\}})$.

\begin{prop}[1(b)(ii)]
Regarding the truth values of the subobject classifier as a Heyting algebra, $\Omega$ has a monoid stucture with multiplication given by the lattice meet operation $\land$.
\end{prop}

\begin{proof}
The meet operation on $\Omega$ is given by set-theoretic intersection. We need to show that this is an associative binary operation on the set of $\{1,2,3\}$-cribles, and that the powerset $\mathbb{P}(\{1,2,3\})$ is the unit of this operation.
\\

\noindent To show that $\land$ is a binary operation, we must check that if $S_1$ and $S_2$ are $\{1,2,3\}$-cribles, then so is $S_1 \land S_2$:
\begin{itemize}
    \item $ S_1 \land S_2 \subseteq S_{\{1,2,3\}}$, because $ S_1 \land S_2 \subseteq S_1 \subseteq S_{\{1,2,3\}}$.
    \item Suppose $c \in S_1 \land S_2$, and $b \subseteq c$.
    \\ Then $b \in S_1$ since $S_1$ is downwards-closed; similarly $b \in S_2$ since $S_2$ is downwards-closed. Therefore $b \in S_1 \cap S_2 = S_1 \land S_2$.
\\ This shows that $S_1 \land S_2$ is really downwards-closed, i.e., a crible.
\end{itemize}

\noindent This binary operation is associative because taking set-theoretic intersections is associative:
\\ for any sets $x,y,z$, we have $(x \cap y) \cap z = x \cap (y \cap z)$.
\\ 

\noindent Finally, the operation has unit $\mathbb{P}(\{1,2,3\})$ because its intersection with any $\{1,2,3\}$-crible $S$ is just $S$.
\\ (After all, we have $S \subseteq S_{\{1,2,3\}} = \mathbb{P}(\{1,2,3\})$.)
\end{proof}
\section{Actions of the subobject classifier}
As a monoid, $\Omega$ can act on a set $X$. This action is well-defined if
$$\langle p=q \rangle \cdot p = \langle p=q \rangle \cdot q$$
for all $p, q \in X$,
where $\langle p=q \rangle $ is the truth value of the assertion $p=q$.
\\

\noindent All of the below concerns well-defined actions as above. We define a partial order on $\Omega$ by $\alpha \leq \beta$ iff $\alpha \land \beta = \alpha$. We may assume that
\begin{equation}\label{alphaassumption}
   \alpha \leq \langle \alpha \cdot p=p\rangle. 
\end{equation}

\noindent We also note that the truth assignment satisfies
$$\langle p=q \rangle \leq \langle q=p \rangle $$
and
$$\langle p=q \rangle \land \langle q=r \rangle \leq \langle p=r \rangle;$$
Goldblatt (2006) provides these as axioms under the section \textit{Heyting-valued sets}.

\begin{lem}[1(c)(i)]
For all $p,q$ the following three statements are equivalent:
\begin{enumerate}
    \item $p = \langle p=q \rangle \cdot p$;
    \item $p = \langle p=q \rangle \cdot q$;
    \item $p = \alpha \cdot q$ for some $\alpha$.
\end{enumerate}
\end{lem}

\begin{proof}
We show that the first statement implies the second, the second implies the third, and the third implies the first.
\begin{itemize}
    \item The first statement implies the second, by assumption of the action being well-defined:
    $$p = \langle p=q \rangle \cdot p =\langle p=q \rangle \cdot q.$$
    \item The second statement implies the third; just take $\alpha = \langle p=q \rangle$.
    \item The third statement implies the first.
    \\ Write $p = \alpha \cdot q$ for some $\alpha$. Since $\alpha \leq \langle \alpha \cdot q=q\rangle = \langle p=q\rangle$, we have $\alpha \land \langle p=q \rangle = \alpha$.
    \\ Therefore,
    
        \begin{align*} 
        p =& \alpha \cdot q \\
        =& (\alpha \land \langle p=q \rangle )\cdot q\\
        =& ( \langle p=q \rangle \land \alpha)\cdot q  && \text{the lattice meet operation is commutative;}\\
        =& \langle p=q \rangle \cdot ( \alpha \cdot q)  && \text{compatibility axiom for monoid actions;}\\
        =& \langle p=q \rangle \cdot p. \\
        \end{align*}
\end{itemize}
\end{proof}

\noindent Write $p \leq q$ if the above equivalent conditions hold.

\begin{lem}[1(c)(ii)]
The relation $\leq$ just defined is a partial order on $X$.
\end{lem}

\begin{proof} We check reflexivity, transitivity and antisymmetry.
\\ Let $p,q,r \in X$.
\begin{itemize}
    \item $p \leq p$.
    \\ Use the third equivalent statement above, and the unit axiom for monoid actions. (This says $p= 1 \cdot p$, where $1$ is the unit of the monoid.)
    \item if $p \leq q$ and $\ q \leq r$, then $p \leq r$.
    \\ Use again the third characterisation of $p \leq q$.
    \\ Writing $p = \alpha \cdot q$ and $q = \beta \cdot r$, we see that $p = \alpha \cdot (\beta \cdot r) = (\alpha \land \beta) \cdot r$.
    \item if $p \leq q$ and $\ q \leq p$, then $p = q$.
    \\ Using the first characterisation of $p \leq q$, write $p = \langle p=q \rangle \cdot p$.
    \\ Using the second characterisation of $q \leq p$, write $q = \langle q=p \rangle \cdot p= \langle p=q \rangle \cdot p$.
    \\ Then we see that $p = \langle p=q \rangle \cdot p = q$.
\end{itemize}
\end{proof}

\begin{prop}[1(c)(iii)]
The action of $\Omega$ on $X$ seen as a map $\Omega \times X \rightarrow X$ is order-preserving in each variable respectively.
\end{prop}

\begin{proof}
\begin{itemize}
We simply check this in each variable.
    \item Suppose we have $\alpha, \beta \in \Omega$ with $\alpha \leq \beta$. We show that $\alpha \cdot p \leq \beta \cdot p$ for each $p \in X$.
    \\ Well, $\alpha = \alpha \land \beta$, so
    $$\alpha \cdot p = (\alpha \land \beta) \cdot p = \alpha \cdot (\beta \cdot p), $$
    and we are done by the third characterisation of $\alpha \cdot p \leq \beta \cdot p$.
    \item Suppose we have $p,q \in X$ with $p \leq q$. We show that $\alpha \cdot p \leq \alpha \cdot q$ for each $\alpha \in \Omega$.
    \\ Well, using the third characterisation of $p \leq q$, there is some $\beta \in \Omega$ such that $p = \beta \cdot q$. Then,
    $$\alpha \cdot p = \alpha \cdot (\beta \cdot q)= (\alpha \land \beta) \cdot q \leq \alpha \cdot q,$$
    where the last step follows from the fact that the action is order-preserving in the first variable, and the fact that $\alpha \land \beta \leq \alpha$.
\end{itemize}
\end{proof}

\begin{prop}[1(c)(iv)]
The partial order on $X$ has a greatest lower bound operation given by $$p \land q := \langle p=q \rangle \cdot p = \langle p=q \rangle \cdot q.$$
\end{prop}

\begin{proof}
This operation is at least well-defined, since we are working with well-defined actions.
\\ We need to show that $p \land q \leq p$ and $p \land q \leq q$, and that $p \land q$ is universal with respect to these properties, meaning that whenever $r \leq p$ and $r \leq q$, we have $r \leq p \land q$.
\begin{itemize}
    \item $p \land q \leq p$.
    \\ $p \land q = \langle p=q \rangle \cdot p$, so w are done by the third characterisation of $p \land q \leq p$. (Just take $\alpha = \langle p \land q \rangle $.)
    \item $p \land q \leq q$.
    \\ Similarly, $p \land q = \langle p=q \rangle \cdot q$, so w are done by the third characterisation of $p \land q \leq q$. (Just take $\alpha = \langle p \land q \rangle $.)
    \item whenever $r \leq p$ and $r \leq q$, we have $r \leq p \land q$.
    \\ We have
\begin{align*} 
r= & \langle r=q \rangle \cdot r &&\text{by the first characterisation of $r \leq q$;} \\
= & \langle r=q \rangle \cdot (\langle r=p\rangle \cdot p) &&\text{by the second characterisation of $r \leq p$;} \\
= & (\langle r=q \rangle \land \langle r=p \rangle) \cdot p &&\text{compatibility axiom for monoid actions;} \\
= & (\langle r=p \rangle \land \langle r=q \rangle) \cdot p &&\text{the lattice meet operation is commutative;} \\
= & (\langle p=r \rangle \land \langle r=q \rangle) \cdot p \\
\leq & (\langle p=q \rangle) \cdot p &&\text{the action is order-preserving in the first variable.}\\
\end{align*}
\end{itemize}
\end{proof}

\begin{prop}[1(c)(v)]
For all $p \in X$ we have an adjunction given by the pair of functors $F_p : \Omega \rightarrow X$ and $G_p : X \rightarrow \Omega$, where
$$F_p(\alpha) = \alpha \cdot p \qquad \text{and} \qquad G_p(q) = \langle p \leq q \rangle.$$
\end{prop}

\begin{proof}
Fix $p \in X$.
\\ It is enough for us to give the unit $\eta : id_\Omega \rightarrow G_pF_p$ and counit $\epsilon : F_p G_p \rightarrow id_X$ of the adjunction.

\noindent Let us first show that $\alpha \leq G_pF_p(\alpha)$ for each $\alpha \in \Omega$, and $F_pG_p(q) \leq q$ for each $q \in X$.

\begin{itemize}
    \item $\alpha \leq G_pF_p(\alpha)$ for each $\alpha \in \Omega$.
    \\ We have
\begin{align*} 
\alpha \leq & \langle p = \alpha \cdot p \rangle &&\text{by assumption (\ref{alphaassumption});}\\
\leq & \langle \langle p = \alpha \cdot p \rangle \cdot p = p \rangle &&\text{by assumption (\ref{alphaassumption});}\\
= & \langle p= \langle p = \alpha \cdot p \rangle \cdot p  \rangle \\
= & \langle p \leq \alpha \cdot p \rangle && \text{by the first characterisation of $ p \leq \alpha \cdot p $; } \\
= &  G_p(\alpha \cdot p)\\
= &  G_pF_p(\alpha)\\
\end{align*}
    \item $F_pG_p(q) \leq q$ for each $q \in X$
    \\We have
\begin{align*} 
F_pG_p(q) = & F_p (\langle p \leq q \rangle)\\
= & \langle p \leq q \rangle \cdot p\\
= & \langle p = p \land q \rangle \cdot p &&\text{by the first or second characterisation of $p \leq q $;} \\
= & \langle p = p \land q \rangle \cdot (p \land q) &&\text{the action is well-defined;} \\
\leq &  p \land q &&\text{by the third characterisation of $\langle p = p \land q \rangle \cdot (p \land q) \leq p \land q $;}\\
\leq & q.
\end{align*}
\end{itemize}

\noindent This means we have a collection of maps $\eta_\alpha : \alpha \rightarrow G_pF_p(\alpha)$ in $\Omega$, and a collection of maps $\epsilon_q : F_pG_p(q) \rightarrow q$ in $X$. These respectively give us our unit $\eta$ and counit $\epsilon$ of the adjunction. Indeed, since all diagrams commute in a poset, we immediately have naturality of $\eta$ and of $\epsilon$, and also that they satisfy the triangle identities
$$\epsilon_{F_p\alpha} \circ F_p\eta_\alpha = id_{F_p \alpha} \qquad \text{and} \qquad G_p\epsilon_q \circ \eta_{G_pq} = id_{G_pq}.$$
\end{proof}

\noindent Write $\langle p \in Y \rangle = \cup_{z \in Y} \langle z=p \rangle$, where $\cup$ is the lattice join operation on $\Omega$, which is just set-theoretic union. (The union of downwards-closed sets is again downwards-closed.)

\begin{prop}[1(c)(vi)]
Any bounded subset $Y$ of $X$ satisfies
$$\text{sup } Y = \langle p \in Y \rangle \cdot p$$
for any upper bound $p$ of $Y$.
\end{prop}

\begin{proof}
We want to show that for each $y \in Y$, we have $y \leq \langle p \in Y \rangle \cdot p$, and furthermore, any upper bound $q$ of $Y$ satisfies $\langle p \in Y \rangle \cdot p \leq q$.

\begin{itemize}
    \item for each $y \in Y$, $y \leq \langle p \in Y \rangle \cdot p$.
    \\ Observe that $y \leq y$ (by reflexivity of $\leq$) and $y \leq p$ (as $p$ is an upper bound for $Y$), so we have $y \leq y \land p$. Therefore,
\begin{align*} 
y \leq &   y \land p \\
= &   \langle y=p \rangle \cdot p \\
\leq & ( \cup_{z \in Y} \langle z=p \rangle ) \cdot p &&\text{the action is order-preserving in the first variable;}\\
= & \langle p \in Y \rangle \cdot p.
\end{align*}

    \item if $y \leq q$ for each $y$, then $( \cup_{z \in Y} \langle z=p \rangle ) \cdot p \leq q$.
    \\ By the proposition above, we know that $\langle p \leq q \rangle \cdot p = F_pG_p(q) \leq q$, so it will be enough to show that $$( \cup_{z \in Y} \langle z=p \rangle ) \cdot p \leq  \langle p \leq q \rangle \cdot p.$$
    In fact, we only need to show that $$\cup_{z \in Y} \langle z=p \rangle \leq \langle p \leq q \rangle ,$$
    since the action is order-preserving in the first variable.
    \\
    
    \noindent Let us now show that $\langle y = p \rangle \leq \langle p \leq q \rangle$ for each $y \in Y$.
    \\ (Then we would be done, by the universal property of the join.)
    \\ First note that $\langle y = y \land q \rangle =1$, since
\begin{align*} 
1 \leq & \langle 1 \cdot y = y \rangle &&\text{by assumption (\ref{alphaassumption});} \\
= & \langle y=y \rangle \\
= & \langle y = y \land q \rangle &&\text{$y = y \land q$ since $y \leq q$;}\\
\leq & 1. &&\text{$1$ is the greatest element of the lattice.} \\
\end{align*}    

Hence, we have
\begin{align*} 
\langle p = y \rangle
= & \langle p = y \rangle  \land 1 \\
= & \langle p = y \rangle  \land \langle y = y \land q \rangle \\
\leq & \langle p = y \land q \rangle, \\
\end{align*}
and so,
\begin{align*}
\langle y = p \rangle
= & \langle p = y \rangle \\
\leq & \langle p = y \rangle  \land \langle y = p \land q \rangle  \\
\leq & \langle p = p \land q \rangle  \\
= & \langle p \leq q \rangle, \\
\end{align*}
as desired.
\end{itemize}
\end{proof}
\section{Monad morphisms}

Recall that for a monad $(T, \mu, \eta)$ on a category $\mathfrak{C}$, the objects of its Eilenberg-Moore category $\mathfrak{C}^T$ are pairs $(A, \sigma_A)$ consisting of an object $A$ in $\mathfrak{C}$ and a $\mathfrak{C}$-morphism $\sigma_A : TA \rightarrow A$ such that
$$\sigma_A \circ T \sigma_A = \sigma_A \circ \mu_A \ \ \ \  \text{and}  \ \ \ \  \sigma_A \circ \eta_A = id_A.$$

\noindent As such, the Eilenberg-Moore category of algebras for the monad comes equipped with a forgetful functor $U: \mathfrak{C}^T \rightarrow \mathfrak{C}$.
\\

\noindent For two categories $\mathfrak{C}_1$ and $\mathfrak{C}_2$, a monad morphism from $T_1$ to $T_2$ is a pair $(F, \theta)$ consisting of a functor $F:\mathfrak{C}_1 \rightarrow \mathfrak{C}_2$ and a natural transformation $\theta: T_2 F \implies FT_1$ such that
$$\theta \circ \mu_2F = F\mu_1 \circ \theta T_1 \circ T_2 \theta \ \ \ \ \ \ \ \text{ and } \ \ \ \ \ \ \ \theta \circ \eta_2F = F\eta_1.$$

\begin{prop}[2(a)(i)]
A monad morphism $(F, \theta)$ as above can be used to uniquely define a functor $\hat{F}$ between the corresponding categories of algebras such that the following diagram commutes:
\begin{equation}\label{2ai}
\begin{tikzcd}
\mathfrak{C}_1^{T_1} \arrow{r}{\hat{F}} \arrow[swap]{d}{U_1} & \mathfrak{C}_2^{T_2} \arrow{d}{U_2} \\
\mathfrak{C}_1 \arrow{r}{F} & \mathfrak{C}_1
\end{tikzcd} 
\end{equation}
\end{prop}

\begin{proof}
Given such a monad morphism $(F, \theta)$, define a functor $\hat{F}: \mathfrak{C}_1^{T_1} \rightarrow \mathfrak{C}_2^{T_2}$ as follows.
\\ Given a $T_1$-algebra $(A, \sigma_A)$, define $\hat{F} (A, \sigma_A) = (FA, \widetilde{\sigma_A})$, where $\widetilde{\sigma_A} =  F \sigma_A \circ \theta_A : T_2 FA \rightarrow FA$.
\\ Indeed, for the desired diagram to commute, the carrier of the algebra $\hat{F} (A, \sigma_A)$ \textit{has} to be $$U_2 \hat{F} (A, \sigma_A) = FU_1 (A, \sigma_A)= FA.$$
Let us check that our definition really gives us a $T_2$-algebra:
\begin{itemize}
    \item
    \[\begin{tikzcd}
T_2^2FA \arrow{r}{(\mu_2)_{ FA}} \arrow[swap]{d}{T_2(F\sigma_A \circ \theta_A)} & T_2FA \arrow{d}{F\sigma_A \circ \theta_A} \\
T_2FA \arrow{r}{F\sigma_A \circ \theta_A} & FA
\end{tikzcd}
\]

\begin{align*} 
& (F\sigma_A \circ \theta_A) \circ T_2(F\sigma_A \circ \theta_A) \\
=&  F\sigma_A \circ (\theta_A \circ T_2F\sigma_A) \circ T_2\theta_A \\
=& F\sigma_A \circ (FT_1\sigma_A\circ \theta_{T_1A}) \circ T_2\theta_A  && \text{by naturality of $\theta$;} \\
=& F(\sigma_A \circ T_1\sigma_A) \circ \theta_{T_1A} \circ T_2\theta_A \\
=& F(\sigma_A \circ (\mu_1)_A) \circ \theta_{T_1A} \circ T_2\theta_A && \text{first algebra axiom for $(A, \sigma_A)$;} \\
=& F\sigma_A \circ (F(\mu_1)_A \circ \theta_{T_1A} \circ T_2\theta_A) \\
=& F\sigma_A \circ (\theta_A \circ (\mu_2)_{ FA}) && \text{$(F,\theta)$ is a monad morphism;} \\
=& (F\sigma_A \circ \theta_A) \circ (\mu_2)_{ FA}. \\
\end{align*}

    \item
\[ \begin{tikzcd}
    FA \arrow{r}{(\eta_2)_{FA}} \arrow[swap]{dr}{id_{FA}} & T_2FA \arrow{d}{F\sigma_A \circ \theta_A
    } \\
     & FA
  \end{tikzcd}
\]

\begin{align*} 
& (F\sigma_A \circ \theta_A) \circ (\eta_2)_{FA} \\
=& F\sigma_A \circ (\theta_A \circ (\eta_2)_{FA}) \\
=& F\sigma_A \circ F(\eta_1)_A && \text{$(F,\theta)$ is a monad morphism;} \\
=& F(\sigma_A \circ (\eta_1)_A) \\
=& F(id_A) && \text{second algebra axiom for $(A, \sigma_A)$;}\\
=& id_{FA} \\
\end{align*}

\end{itemize}

\noindent We still have to define $\hat{F}$ on morphisms in $\mathfrak{C_1}^{T_1}$.
\\ Given a homomorphism of $T_1$-algebras $(A,\sigma_A) \xrightarrow{h} (B,\sigma_B)$, define $\hat{F}(h): (FA, F\sigma_A \circ \theta_A) \rightarrow (FB, F\sigma_B \circ \theta_B)$ as the homomorphism of $T_2$-algebras $FA \xrightarrow{Fh} FB$.
\\ Indeed, for the desired diagram to commute, our choice of $\hat{F}(h)$ is defined uniquely: $$U_2 \hat{F} (h) = FU_1 (h)= Fh.$$
\\ Let us check that $Fh$ is really a homomorphism of $T_2$-algebras:

\[\begin{tikzcd}
T_2FA \arrow{r}{F\sigma_A \circ \theta_A} \arrow[swap]{d}{T_2Fh} & FA \arrow{d}{Fh} \\
T_2FB \arrow{r}{F\sigma_B \circ \theta_B} & FB
\end{tikzcd}
\]

\begin{align*} 
& Fh \circ (F\sigma_A \circ \theta_A) \\
=& F(h \circ \sigma_A) \circ \theta_A \\
=& F(\sigma_B \circ T_1h) \circ \theta_A  && \text{$h$ is a homomorphism of $T_1$-algebras;} \\
=& F\sigma_B \circ (FT_1h \circ \theta_A) \\
=& F\sigma_B \circ (\theta_B \circ T_2Fh)  && \text{naturality of $\theta$;} \\
=& (F\sigma_B \circ \theta_B) \circ T_2. \\
\end{align*}

Finally, let us check the two functoriality axioms:

\begin{itemize}
    \item for each $T_1$-algbera $(A,\sigma_A)$, we have $\hat{F}(id_{(A,\sigma_A)}) = F(id_A) = id_{FA} =id_{\hat{F}(A, \sigma_A)} $;
    \item if $(A,\sigma_A) \xrightarrow{h} (B,\sigma_B)$ and $(B,\sigma_B) \xrightarrow{k} (C,\sigma_C)$ are maps in $\mathfrak{C_1}^{T_1}$, then $$\hat{F}(k \circ h) = F(k \circ h) = F(k) \circ F(h) = \hat{F}(k) \circ \hat{F}(h).$$
\end{itemize}
\end{proof}

\begin{prop}[2(a)(ii)]
The converse also holds: each commutative diagram \eqref{2ai} gives rise to a natural transformation $\theta$ making $(F, \theta)$ into a monad morphism.
\end{prop}

\begin{proof}
Suppose we have a functor $\hat{F}: \mathfrak{C_1}^{T_1} \rightarrow \mathfrak{C_2}^{T_2}$ such that $FU_1 = U_2 \hat{F}$. We construct a natural transformation $\theta:T_2F \implies FT_1$ as follows. Apply $\hat{F}$ to the free $T_1$-algebra $(T_1A, (\mu_1)_A)$ to get a $T_2$-algebra $(FT_1A, \widetilde{(\mu_1)_A})$. Then, set
$$\theta_A = \widetilde{(\mu_1)_A} \circ T_2F((\eta_1)_A):T_2FA \rightarrow FT_1A$$ for each $A \in \mathfrak{C_1}$.
\\

\noindent Note that for each homomorphism of $T_1$-algebras $h$, we again know what $\hat{F}(h)$ must be. It is given by $$U_2 \hat{F} (h) = FU_1 (h)= Fh.$$
\noindent Now, let us check naturality of $\theta$: for each $A \xrightarrow{f} B$ in $\mathfrak{C_1}$,

\[\begin{tikzcd}
T_2FA \arrow{r}{T_2Ff} \arrow[swap]{d}{\theta_A} & T_2FB \arrow{d}{\theta_B} \\
FT_1A \arrow{r}{FT_1f} & FT_1B
\end{tikzcd}
\]

\begin{align*} 
& FT_1f \circ \theta_A \\
=& (FT_1f \circ \widetilde{(\mu_1)_A}) \circ T_2F((\eta_1)_A)\\
=& (\widetilde{(\mu_1)_B} \circ T_2FT_1f ) \circ T_2F((\eta_1)_A) && \text{$T_1f$ is a hom of (free) $T_1$-algebras, so $\hat{F}T_1f=FT_1f$ is a hom of $T_2$-algebras;}\\
=& \widetilde{(\mu_1)_B} \circ T_2F(T_1f  \circ (\eta_1)_A) \\
=& \widetilde{(\mu_1)_B} \circ T_2F((\eta_1)_B  \circ f)  && \text{naturality of $\eta_1$;}\\
=& \theta_B \circ T_2Ff
\end{align*}

\noindent Next, let us verify that $(F,\theta)$ is indeed a monad morphism:

\begin{itemize}
    \item $\theta \circ \mu_2F = F\mu_1 \circ \theta T_1 \circ T_2 \theta$:
    \begin{align*} 
& \theta_A \circ (\mu_2)_{FA} \\
=& \widetilde{(\mu_1)_A} \circ T_2F((\eta_1)_A) \circ (\mu_2)_{FA} \\
=& \widetilde{(\mu_1)_A} \circ  (\mu_2)_{FT_1A} \circ T_2^2F((\eta_1)_A)  && \text{naturality of $\mu_2$;}\\
=& \widetilde{(\mu_1)_A} \circ T_2 \widetilde{(\mu_1)_A} \circ T_2^2F((\eta_1)_A) && \text{first algebra axiom for $(FT_1A, \widetilde{(\mu_1)_A})$;}\\
=& \widetilde{(\mu_1)_A} \circ T_2F((\mu_1)_A  \circ (\eta_1)_{T_1A}) \circ T_2 \widetilde{(\mu_1)_A} \circ T_2^2F((\eta_1)_A) && \text{$(\mu_1)_A  \circ (\eta_1)_{T_1A}=id_{T_1A}$, by the unit monad axiom}\\
& && \text{for $T_1$;}\\
=& \widetilde{(\mu_1)_A} \circ T_2F((\mu_1)_A)  \circ T_2F((\eta_1)_{T_1A}) \circ T_2 \widetilde{(\mu_1)_A} \circ T_2^2F((\eta_1)_A)\\
=& F((\mu_1)_A) \circ \widetilde{(\mu_1)_{T_1A}} \circ T_2F((\eta_1)_{T_1A}) \circ T_2 \widetilde{(\mu_1)_A} \circ T_2^2F((\eta_1)_A)  && \text{$\hat{F}((\mu_1)_A)=F((\mu_1)_A)$ is a hom of $T_2$-algebras,}\\
& && \text{because $(\mu_1)_A$ is a hom of (free) $T_1$-algebras,}\\
& && \text{by the associativity monad axiom for $T_1$;}\\
=& F((\mu_1)_A) \circ \theta_{T_1A} \circ T_2 \theta_A
\end{align*}
    \item $\theta \circ \eta_2F = F\eta_1$:
    \begin{align*} 
& \theta_A \circ (\eta_2)_{FA} \\
=& \widetilde{(\mu_1)_A} \circ T_2F((\eta_1)_A) \circ (\eta_2)_{FA} \\
=& \widetilde{(\mu_1)_A} \circ (\eta_2)_{FT_1A} \circ F((\eta_1)_A) && \text{naturality of $\eta_2$;}\\
=& id_{FT_1A} \circ F((\eta_1)_A) && \text{second algebra axiom for $(FT_1A, \widetilde{(\mu_1)_A})$;}\\
=& F((\eta_1)_A) \\
\end{align*}
\end{itemize}

\end{proof}

\begin{prop}[2(a)(iii)]
If $F$ is faithful so is $\hat{F}$.
\end{prop}

\begin{proof}
Let $F$ be faithful, meaning given any two objects $A,B$ of $\mathfrak{C}_1$ and a map $FA \xrightarrow{g} FB$ in $\mathfrak{C}_2$, there is at most one map $A \xrightarrow{f} B$ in $\mathfrak{C}_1$ such that $g=Ff$.
\\

\noindent Suppose we are given two $T_1$-algebras $(A, \sigma_A), (B, \sigma_B)$, a homomorphism of $T_2$-algebras $(FA, \widetilde{\sigma_A}) \xrightarrow{k} (FB, \widetilde{\sigma_B})$, and two homomorphisms of $T_1$-algebras $(A, \sigma_A) \xrightarrow{h_1,h_2} (B, \sigma_B)$ such that $\hat{F}h_1=k=\hat{F}h_2$. We must show that $h_1=h_2$.
\\ Well, $U_2k$ is a map $FA \rightarrow FB$ in $\mathfrak{C}_2$, so by faithfulness of $F$ there is at most of map $A \xrightarrow{f} B$ in $\mathfrak{C}_1$ such that $Ff=U_2k$. However, both $f=h_1$ and $f=h_2$ satisfy this equation:
$$Fh_i = \hat{F}h_i = k = U_2k \ \ (i=1,2).$$
Therefore, we must have $h_1=h_2$ (as maps in $\mathfrak{C}_1$, so also as maps in $\mathfrak{C}_1^{T_1}$).
\end{proof}

\begin{prop}[2(a)(iv)]
If $F$ is fully faithful and each component of $\theta$ is an epimorphism, then $\hat{F}$ is fully faithful.
\end{prop}

\begin{proof}
Now we additionally assume $F$ is full, meaning given any two objects $A,B$ of $\mathfrak{C}_1$ and a map $FA \xrightarrow{g} FB$ in $\mathfrak{C}_2$, there is some map $A \xrightarrow{f} B$ in $\mathfrak{C}_1$ such that $g=Ff$.
\\

\noindent Suppose we are given two $T_1$-algebras $(A, \sigma_A), (B, \sigma_B)$ and a homomorphism of $T_2$-algebras $$(FA, \widetilde{\sigma_A}) \xrightarrow{k} (FB, \widetilde{\sigma_B}).$$
\\ By fullness of $F$, there is some map $A \xrightarrow{f} B$ in $\mathfrak{C}_1$ such that $U_2k=Ff : FA \rightarrow FB$.
\\

\noindent Let us show that $f$ is a homomorphism of $T_1$-algebras:

\[\begin{tikzcd}
T_1A \arrow{r}{\sigma_A} \arrow[swap]{d}{T_1f} & A \arrow{d}{f} \\
T_1B \arrow{r}{\sigma_B} & B
\end{tikzcd}
\]

\noindent Well, we have
\begin{align*} 
& F(f \circ \sigma_A) \circ \theta_A \\
=& Ff \circ F(\sigma_A) \circ \theta_A \\
=& F(\sigma_B) \circ \theta_B \circ T_2Ff \\
=& F(\sigma_B) \circ FT_1f \circ \theta_A \\
=& F(\sigma_B \circ T_1f) \circ \theta_A \\
\end{align*}

\noindent Since by assumption $\theta_A$ is epic, this implies $F(f \circ \sigma_A) =  F(\sigma_B \circ T_1f)$.
\\ Since we also assumed $F$ to be faithful, we have $f \circ \sigma_A = \sigma_B \circ T_1f$, as desired.
\\

\noindent By construction, $\hat{F}f=Ff=k$, so we have completed our proof that $\hat{F}$ is full.
\\ We already know by the result above that $\hat{F}$ is faithful, so we are done.

\end{proof}

\iffalse
\begin{thm}[2(a)(v)]
If $F$ is fully faithful and $\theta$ is a split epimorphism, then $\hat{F}$ is the pullback of $F$ along $U_2$.
\end{thm}

Since $\theta$ is split epic, there is some natural transformation $\beta : FT_1 \implies T_2F$ such that $\theta \circ \beta = id_{FT_1}$. In particular, since $$\theta_A \circ \beta_A = id_{FT_1A},$$ each component $\theta_A$ is split epic, hence epic!
Therefore we are in the situation of Proposition 2(a)(iv) above, and $\hat{F}$ is fully faithful.
\\

\fi

\begin{thm}[2(a)(v)]
If $F$ is fully faithful and $\theta$ is an isomorphism, then $\hat{F}$ is the pullback of $F$ along $U_2$.
\end{thm}

\begin{proof}
Since $\theta$ is an isomorphism, there is some natural transformation $\beta : FT_1 \implies T_2F$ such that $\theta \circ \beta = id_{FT_1}$ and $\beta \circ \theta = id_{T_2F}$. In particular, for each $A$ we have $$\theta_A \circ \beta_A = id_{FT_1A},$$ hence each component $\theta_A$ is split epic, hence epic!
Therefore we are in the situation of Proposition 2(a)(iv) above, and $\hat{F}$ is fully faithful.
\\

\noindent We wish to show that

\[ \begin{tikzcd}
\mathfrak{C}_1^{T_1} \arrow{r}{\hat{F}} \arrow[swap]{d}{U_1} & \mathfrak{C}_2^{T_2} \arrow{d}{U_2} \\
\mathfrak{C}_1 \arrow{r}{F} & \mathfrak{C}_1
\end{tikzcd} 
\]

\noindent is a pullback square.
\\ It certainly commutes, so let us check that it is universal as such.
\\ Suppose we have another commutative square:

\[ \begin{tikzcd}
\mathfrak{D} \arrow{r}{G_2} \arrow[swap]{d}{G_1} & \mathfrak{C}_2^{T_2} \arrow{d}{U_2} \\
\mathfrak{C}_1 \arrow{r}{F} & \mathfrak{C}_1
\end{tikzcd} 
\]

\noindent We construct a unique functor $G: \mathfrak{D} \rightarrow \mathfrak{C}_1^{T_1}$ such that
$$U_1G=G_1 \qquad \text{and} \qquad \hat{F} G = G_2.$$

\begin{itemize}
    \item \textbf{Uniqueness of $G$}:
    \\ Suppose $G$ is a functor satisfying these conditions.
    \\ Fix $D \in \mathfrak{D}$. The carrier of the $T_1$-algebra $GD$ \textit{has} to be $G_1D$, since $U_1GD = G_1D$. Hence, write
    $$G(D) = (G_1D, \ T_1G_1D \xrightarrow{\sigma_{G_1D}} G_1D).$$
    We claim that $\sigma_{G_1D}$ is also uniquely determined.
    Apply $\hat{F}$ to this $T_1$-algebra, to get the $T_2$-algebra
    $$\hat{F}G D = G_2D = (U_2G_2D, \ T_2FG_1D \xrightarrow{\widetilde{\sigma_{G_1D}}} FG_1D),$$
    where $\widetilde{\sigma_{G_1D}} = F \sigma_{G_1D} \circ \theta_{G_1D}$.
    \\ Then, $\widetilde{\sigma_{G_1D}} \circ \beta_{G_1D}= F \sigma_{G_1D} \circ \theta_{G_1D} \circ \beta_{G_1D} =  F \sigma_{G_1D}$.
    \\
    
    This shows that $\sigma_{G_1D}$ is unique!
    \\ ($F$ is fully faithful, so there can only be one map $F^{-1}(\widetilde{\sigma_{G_1D}} \circ \beta_{G_1D})$ sent by $F$ to $\widetilde{\sigma_{G_1D}} \circ \beta_{G_1D}$, and we have just shown that $\sigma_{G_1D}$ is such a map.)
    \\
    
    We have shown that the action of $G$ on objects is uniquely determined. Next, let us show that its action on maps is also uniquely determined.
    \\ Let $D \xrightarrow{f} D'$ be a map in $\mathfrak{D}$. Then $\hat{F} G f = G_2 f$. This shows that $Gf$ is unique.
    \\ ($\hat{F}$ is fully faithful, so there can only be one map $\hat{F}^{-1}(G_2f)$ sent by $\hat{F}$ to $G_2f$, and we have just shown that $Gf$ is such a map.)
    
    \item \textbf{Existence of $G$}:
    \\ We know by the above what $G$ has to be, if it exists:
    \\ It must be given by $D \mapsto (G_1D, \ T_1G_1D \xrightarrow{\sigma_{G_1D} = F^{-1}(\widetilde{\sigma} \circ \beta_{G_1D})} G_1D)$, $(D \xrightarrow{f} D') \mapsto (GD \xrightarrow{\hat{F}^{-1}(G_2f)} GD')$,
    \\ where $\widetilde{\sigma}$ is the map given by $G_2D = (U_2G_2D, \ \widetilde{\sigma}) =(FG_1D, \widetilde{\sigma})$.
    \\ Let us show that this gives us a well-defined functor from $\mathfrak{D}$ to $\mathfrak{C}_1^{T_1}$.
    \\
    
    Note that $Gf=\hat{F}^{-1}(G_2f)$ is a homomorphism of $T_1$-algebras, \textit{by definition}. (It is the one that maps under $\hat{F}$ to the homomorphism of $T_2$-algebras $G_2f$.
    \\ Next, we see that $GD$ is really a $T_1$-algebra:
    \begin{itemize}
        \item $\sigma_{G_1D} \circ T_1 \sigma_{G_1D} = \sigma_{G_1D} \circ (\mu_1)_{G_1D}$:
        \\ Since $F$ is faithful, it is enough to check equality on $F$ applied to these maps.
        \\ We have
        \begin{align*}
            F(\sigma_{G_1D} \circ T_1 \sigma_{G_1D})
            =& F\sigma_{G_1D} \circ FT_1 \sigma_{G_1D} \\
            =& \widetilde{\sigma} \circ \beta_{G_1D} \circ FT_1 \sigma_{G_1D} \\
            =& \widetilde{\sigma} \circ T_2F\sigma_{G_1D} \circ \beta_{T_1G_1D} &&\text{by naturality of $\beta$;}\\
            =& \widetilde{\sigma} \circ T_2(\widetilde{\sigma} \circ \beta_{G_1D}) \circ \beta_{T_1G_1D} \\
            =& \widetilde{\sigma} \circ T_2\widetilde{\sigma} \circ T_2\beta_{G_1D} \circ \beta_{T_1G_1D} \\
            =& \widetilde{\sigma} \circ (\mu_2)_{FG_1D} \circ T_2\beta_{G_1D} \circ \beta_{T_1G_1D}\\
            =& \widetilde{\sigma} \circ \beta_{G_1D} \circ \theta_{G_1D} \circ (\mu_2)_{FG_1D} \circ T_2\beta_{G_1D} \circ \beta_{T_1G_1D} \\
            =& \widetilde{\sigma} \circ \beta_{G_1D} \circ F ((\mu_1)_{G_1D}) \circ \theta_{T_1G_1D} \circ T_2 \theta_{G_1D} \circ T_2\beta_{G_1D} \circ \beta_{T_1G_1D} &&\text{$\theta$ is a monad morphism;}\\
            =& \widetilde{\sigma} \circ \beta_{G_1D} \circ F ((\mu_1)_{G_1D}) \circ \theta_{T_1G_1D} \circ T_2 (\theta_{G_1D} \circ \beta_{G_1D}) \circ \beta_{T_1G_1D} \\
            =& \widetilde{\sigma} \circ \beta_{G_1D} \circ F ((\mu_1)_{G_1D}) \circ \theta_{T_1G_1D} \circ  \beta_{T_1G_1D} &&\text{$(\theta \circ \beta)_{G_1D} = id_{FT_1G_1D}$;}\\
            =& \widetilde{\sigma} \circ \beta_{G_1D} \circ F ((\mu_1)_{G_1D})  &&\text{$(\theta \circ \beta)_{T_1G_1D} = id_{FT_1^2G_1D}$;}\\
            =& F\sigma_{G_1D} \circ F ((\mu_1)_{G_1D}) \\
            =& F(\sigma_{G_1D} \circ (\mu_1)_{G_1D}).
        \end{align*}
        
        \item $\sigma_{G_1D} \circ (\eta_1)_{G_1D} = id_{G_1D}$:
        \\ Again it is enough to check equality on $F$ applied to these maps. We have
        \begin{align*}
            F(\sigma_{G_1D} \circ (\eta_1)_{G_1D})
            =& F\sigma_{G_1D} \circ F((\eta_1)_{G_1D}) \\
            =& \widetilde{\sigma} \circ \beta_{G_1D} \circ F((\eta_1)_{G_1D}) \\
            =& \widetilde{\sigma} \circ \beta_{G_1D} \circ \theta_{G_1D}
            \circ (\eta_2)_{FG_1D} &&\text{$\theta$ is a monad morphism}\\
            =& \widetilde{\sigma} \circ (\eta_2)_{FG_1D} \\
            =& id_{FG_1D} &&\text{second algebra axiom for $G_2D =(FG_1D, \widetilde{\sigma})$} \\
            =& F(id_{G_1D}). \\
        \end{align*}
    \end{itemize}

    Next, let us check functoriality of $G$:
    \begin{itemize}
        \item $G(id_D) = \hat{F}^{-1}(G_2 id_D)= \hat{F}^{-1}(id_{G_2D})= id_{GD}$,
        \\ where in the last step we are using that $\hat{F}$ is fully faithful, and that $\hat{F}id_{GD}= id_{\hat{F}GD} = id_{G_2D}$.
        \item Let $D \xrightarrow{f} D'$, $D' \xrightarrow{g} D''$ be maps in $\mathfrak{D}$.
        \\ Then $G(g \circ f) = \hat{F}^{-1}(G_2(g \circ f))= Gg \circ Gf$,
        \\ where in the last step we are using that $\hat{F}$ is fully faithful, and that
        $$\hat{F} (Gg \circ Gf) = \hat{F} Gg \circ \hat{F} Gf = G_2g \circ G_2 f = G_2(g \circ f).$$
    \end{itemize}
    
    Finally, we must check the universal property for pullbacks:
     \begin{itemize}
        \item $\hat{F} G = G_2$:
        \begin{align*}
            \hat{F}GD =&\hat{F}(G_1D, \ F^{-1}(\widetilde{\sigma} \circ \beta_{G_1D})) \\
            =&(FG_1D, \ F(F^{-1}(\widetilde{\sigma} \circ \beta_{G_1D}) ) \circ \theta_{G_1D}) \\
            =&(FG_1D, \ \widetilde{\sigma} \circ \beta_{G_1D}  \circ \theta_{G_1D}) \\
            =&(FG_1D, \ \widetilde{\sigma} )\\
            =& G_2D, \\
        \end{align*}
        so we have equality on objects.
        \\ On maps,
        $$\hat{F}G (D \xrightarrow{f} D') = \hat{F} \hat{F}^{-1}(G_2f) = G_2f.$$
        \\
        
        \item $U_1G = G_1$:
        \begin{align*}
            U_1G(D) = G_1D, \\
        \end{align*}
        so we have equality on objects.
        \\ On maps,
        $$U_1G(D \xrightarrow{f} D') = U_1 \hat{F}^{-1}(G_2f) = G_1f,$$
        where the last equality follows from the equality of $F$ applied to these maps:
        $$FU_1 \hat{F}^{-1}(G_2f) = U_2 \hat{F} \hat{F}^{-1}(G_2f)= U_2G_2f= FG_1f.$$
        
    \end{itemize}
    
\end{itemize}
\end{proof}
\newpage
\section{Modelling non-determinism with monads}
\begin{essay}[2(b)] Monad are commonly used to model non-deterministic procedures. This can be implemented effectively in Haskell (or indeed other functional programming languages).
\\

\noindent Non-determinism of an algorithm just means that at each stage there are several possible outputs that can be taken as input of the next stage. For instance, non-deterministic finite state automata have transition functions of the form
$\delta : Q \times \Sigma \cup \epsilon \rightarrow \mathcal{P} (Q)$, where $Q$ is the set of states, and $\Sigma$ is the alphabet. At each stage the machine may transition to several (or no) states.

\noindent We give three brief instances of how monads model non-determinism. Since we would like to account for each possible branch of the computation, we use the list monad to form a collection of the outcomes. (The Powerset monad works fine, too --- see our closing remark.)
\\

\noindent Our first example demonstrates the extent to which monads are inbuilt into Haskell.
\\ Suppose we have a list of functions $[f_1, \cdots, f_n]$ on the same domain, and a list of elements $[x_1, \cdots, x_m]$ in the domain. We would like to evaluate each function on each element. We may execute these calculations `in parallel' with the command

\begin{lstlisting}
        ghci> (app) <$> [f1, ..., fn] <*> [x1, ..., xm]
\end{lstlisting}
where {\tt app} takes a function and an element and applies the former to the latter.
\\ The monad does not even appear explicitly, but we are actually using the list monad as an applicative.
\\ To make the monad structure more clear, we note that this is nothing more than the list comprehension
        $$\text{{\tt [f app x | f <- [ f1,..., fn ], x <- [x1,..., xm]]}},$$
\noindent In general, a comprehension has the form {\tt [t|q]}, where {\tt t } is a term, and {\tt q} a qualifier. A qualifier has one of the following forms:

\begin{itemize}
    \item the empty qualifier $\Lambda$;
    \item a generator \texttt{x $\leftarrow$ u}, for some variable \texttt{x} and list-valued term \texttt{\tt u};
    \item a composition {\tt (p,q)} of shorter qualifiers.
\end{itemize}

\noindent The point is that monads $(T, \eta, \mu)$ can be derived from comprehensions, and vice versa. We have:
\begin{itemize}
    \item {\tt [t|$\Lambda$]} = $\eta_t$;
    \item {\tt [t|x $\leftarrow$ u]} = $T(\lambda x. t)u$;
    \item {\tt [t|(p,q)]} = $\mu(${\tt [[t|q]|p]}$)$.
\end{itemize}
\noindent See Wadler (1992) for further details.
\\

\noindent A second example is the composition of multi-valued functions. If we take the $m^{th}$ root of a complex number, and then take the $n^{th}$ root of the result, we want this to be equal to the result of taking the $mn^{th}$ root. Hence we would like to return $mn$ possibilities in a list.
\\ Once again the list monad achieves this. We simply define:

        \begin{lstlisting}
        bind :: (Complex Double -> [Complex Double]) -> ([Complex Double] -> [Complex Double])
        bind f x = concat (map f x)
        
        unit :: Complex Double -> [Complex Double]
        unit x = [x]
        \end{lstlisting}

\noindent Here we have given the monad in Kleisli form, where the Kleisli extension is given by {\tt bind}, and the unit by {\tt unit}, of course.
\\

\noindent We end with another application of the list monad, to the modelling of a conditional probability problem (Taylor, 2013). Suppose we are given two boxes $A$ and $B$, each containing three marbles. The first has one white and two black marbles; the second has all three marbles white. We blindly select a box at random, and then from the box randomly extract a marble. If this is white, what is the probability that we selected the first box?
\\ Using Bayes' Theorem we may calculate this probabillity to be $\frac{(1/3)(1/2)}{4/6}=1/4$, but using the list monad we may model the scenario explicitly; such methods are valuable in applied statistics.

        \begin{lstlisting}
        
            data Box = BoxA | BoxB deriving (Show)
            
            data Marble = Black | White deriving (Eq,Show)
            
            extract BoxA   = [White, Black, Black]
            extract BoxB = [White, White, White]
            
            pick = [BoxA, BoxB]
            
            trial = do
              box   <- pick         -- Simulate picking a box at random
              result <- toss coin    -- Extract a marble and observe the result
              guard (result == White) -- We only proceed if the marble is white
              return box            -- Return which box this marble came from
              
              
            >> trial
        \\    [BoxA, BoxB, BoxB, BoxB]
        \end{lstlisting}

\noindent Let us briefly explain the code. First, we defined two data types that store the values of the boxes and the marbles. Then, we modelled the outcomes using a list. We next defined a function {\tt pick} that modelled the random selection of either box.
\\ Where the monad comes in is the {\tt do} block. 
\\ In Haskell,
        \begin{lstlisting}
        do { x1 <- action1
           ; x2 <- action2
           ; mk_action3 x1 x2 }
        \end{lstlisting}
is short for
        
        \begin{lstlisting}
        action1 >>= (\ x1 ->
          action2 >>= (\ x2 ->
            mk_action3 x1 x2 ))
        \end{lstlisting}

\noindent where {\tt >>=} is the bind combinator, i.e., Kleisli extension:
$$\text{{\tt (mx >>= f) : (TA, A $\rightarrow$ TB) $\rightarrow$ TB}}$$

\noindent Since the program outputs three instances of {\tt BoxB} and one of {\tt BoxA}, we conclude that the probability that we had chosen Box $A$, given that we have extracted a white marble, is $1/4$ as expected.

\noindent It is worth noting that the algorithm is not really randomly selecting boxes and extracting marbles; we are merely simulating the non-determinism of the process by listing all the outcomes --- that is the point, after all.
\\

\noindent As a closing remark, we note that the powerset monad also models non-determinism effectively, albeit with less structure than the list monad --- now we no longer keep track of duplicate entries, and elements of a list are returned having no order. Depending on the scenario it is useful to use one monad or the other. For instance, the powerset monad would not model our last example, since there we needed to keep track of duplicates. However, the powerset monad might better suit our first example if we wished to return values in no particular order and without repetitions.

\end{essay}
\newpage
\appendix
\section{Appendix: Propositional Logic, Intuitionistic Logic, Classical Logic}

\noindent The formal language \textbf{PL} (propositional logic) is described by its alphabet and formation rules:
\begin{itemize}
    \item the alphabet for \textbf{PL} consists of:
    \begin{enumerate}
        \item a collection $\{ \pi_i : i \in \mathbb{N} \}$ of symbols, called the \textbf{propositional variables};
        \item the symbols $\sim, \land, \lor, \supset$;
        \item the bracket symbols $)$, $($.
    \end{enumerate}
    \item we have the following formation rules for \textbf{PL}-sentences:
    \begin{enumerate}
        \item each propositional variable $\pi_i$ is a \textbf{PL}-\textbf{sentence}, or \textbf{formula};
        \item if $\alpha$ is a sentence, so is $(\sim \alpha)$;
        \item if $\alpha$ and $\beta$ are sentences, so are $(\alpha \land \beta), (\alpha \lor \beta), (\alpha \supset \beta)$.
    \end{enumerate}
\end{itemize}
\noindent Define $\Phi_0 := \{\pi_0, \pi_1, \pi_2, \ldots \}$ and $\Phi := \{ \alpha : \alpha \text{ is a \textbf{PL}-sentence} \}$.
\\

\noindent An \textbf{axiom system} is described by a collection of sentences (its \textbf{axioms}), and a collection of \textbf{inference rules}, which prescribe operations on sentences in order to derive new ones. A \textbf{proof sequence} is a finite sequence of sentences, each of which is either an axiom or derivable from earlier members of the sequence using an inference rule.
\\

\noindent The axiom system \textbf{CL} (classical logic) has axioms which are sentences of one of the form:
\begin{enumerate}
    \item $\alpha \supset (\alpha \land \alpha)$;
    \item $(\alpha \land \beta) \supset (\beta \land \alpha)$;
    \item $(\alpha \supset \beta) \supset ((\alpha \land \gamma) \supset (\beta \land \gamma))$;
    \item $((\alpha \supset \beta) \land (\beta \supset \gamma)) \supset (\alpha \supset \gamma))$;
    \item $\beta \supset (\alpha \supset \beta)$;
    \item $(\alpha \land (\alpha \supset \beta)) \supset \beta$;
    \item $\alpha \subset (\alpha \lor \beta)$;
    \item $(\alpha \lor \beta) \supset (\beta \lor \alpha)$;
    \item $((\alpha \supset \gamma) \land (\beta \supset \gamma)) \supset ((\alpha \lor \beta) \supset \gamma)$;
    \item $(\sim \alpha) \supset (\alpha \supset \beta)$;
    \item $((\alpha \supset \beta) \land (\alpha \supset \sim \beta)) \supset (\sim \alpha)$;
    \item $\alpha \lor (\sim \alpha)$.
\end{enumerate}
\noindent The system \textbf{CL} has one inference rule:
\\ \textbf{Modus Ponens.} From $\alpha$ and $\alpha \supset \beta$ we may derive $\beta$.
\\

\noindent We say $\alpha$ is a \textbf{CL}-\textbf{theorem}, and write $\vdash_{CL} \alpha$, if $\alpha$ is the last member of some proof sequence in \textbf{CL}.
\\

\noindent That last axiom $\alpha \lor (\sim \alpha)$ in \textbf{CL} is called the \textit{law of excluded middle}; if we remove this axiom, keeping all the other axioms and the Modus Ponens inference rule, then we obtain the axiom system \textbf{IL} (intuitionistic logic).
\\

\noindent \noindent We say $\alpha$ is an \textbf{IL}-\textbf{theorem}, and write $\vdash_{IL} \alpha$, if $\alpha$ is the last member of some proof sequence in \textbf{IL}.
\\

\end{document}